\documentclass[12pt]{amsart}
\usepackage{amssymb,amsfonts,amsmath,amsopn,amstext,amscd,latexsym, amsthm, enumerate,mathrsfs}
\usepackage{color}
\usepackage{xcolor}
\usepackage{hyperref}
\hypersetup{
     colorlinks   = true,
     citecolor    = blue
}

\usepackage[margin=30mm]{geometry}
\usepackage{bbm}
\usepackage[all,cmtip]{xy}

\usepackage{enumerate}
\usepackage[shortlabels]{enumitem}

\newtheorem{theorem}{Theorem}[section]
\newtheorem{thm}[theorem]{Theorem}

\newtheorem{cor}[theorem]{Corollary}

\theoremstyle{remark}

\numberwithin{equation}{section}

\begin{document}

\title[Super connectivity  of lexicographic product graphs]{Super connectivity  of lexicographic product graphs}

\author[Kh~Kamyab]{Khalid Kamyab}
\address{Department of Mathematics,  Urmia University,  Urmia 57135, Iran}
\email{kh.kamyab@urmia.ac.ir}

\author[M.~Ghasemi]{Mohsen Ghasemi}
\address{Department of Mathematics, Urmia University,  Urmia  57135, Iran}
\email{m.ghasemi@urmia.ac.ir}

\author[R.~Varmazyar]{Rezvan Varmazyar}
\address{Department of Mathematics,  Khoy Branch, Islamic Azad University, Khoy, Iran}
\email{varmazyar@iaukhoy.ac.ir}


\thanks{}


\date{\today}



\begin{abstract}
For a graph $G$, $k(G)$ denotes its connectivity. A graph is super
connected if every minimum vertex-cut isolates a vertex. Also
$k_{1}$-connectivity of a connected graph is the minimum number of
vertices whose deletion gives a disconnected graph without isolated
vertices. This paper provides bounds for the super connectivity and
$k_{1}$-connectivity of the lexicographic product of two graphs.
\end{abstract}

\maketitle
\section{Introduction}
Let $G=(V(G), E(G))$ be a simple undirected graph without loops and
multiple edges. We follow Bondy and Murty [3] for terminologies and
notations not defined here.\\For $x\in V(G)$, the {\it neighborhood} of $x$
is $N_{G}(x)=\{y\in V(G) \mid xy\in E(G)\}$ and $d_{G}(x)=|N_{G}(x)|$.
If $d_{G}(x)=0$ then $x$ is called an {\it isolated vertex}. We denote the
set of isolated vertices of a graph $G$ by $V_0(G)$. The minimum
degree of $G$ is $\delta(G)={\rm min} \{d_{G}(x) \mid x\in V(G)\}$. The
complete graph with $n$ vertices is denoted by $K_{n}$.

 Let $S\subseteq V(G)$. $S$ is called a {\it vertex-cut} of $G$
if $G-S$ is disconnected or reduces $G$ to the trivial graph
$K_{1}$.

The {\it connectivity} of the graph $G$, $k(G)$, is defined as the minimum
cardinality $|S|$ where $S$ is a vertex-cut of $G$. Clearly,
$k(G)\leq \delta(G)$ and $k(G)=0$ if and only if $G$ is disconnected
or $G=K_1$.

An interconnection network is often modeled as a graph $G=(V(G),
E(G))$, where $V(G)$ is the set of processors and $E(G)$ is the set
of communication links in the network. The connectivity is an
important measure of the stability of any network and gives the
minimum cost to disrupt the fail at the same time. For more details we refere the reder to [2].

A vertex-cut $S$ of the graph $G$ is called {\it $k_{1}$-vertex-cut} if
$G-S$ contains no isolated vertices. The {\it $k_{1}$-connectivity} of the
graph $G$, $k_{1}(G)$, is defined as the minimum cardinality $|S|$
where $S$ is a $k_{1}$-vertex-cut of $G$. The $k_{1}$-connectivity
of a graph was proposed in [9]. We write $k_{1}(G)=\infty$ if
$k_{1}(G)$ does not exist, for example, $k_{1}(K_{1,n})=\infty$.

Also, a graph $G$ is {\it super connected} if every minimum vertex-cut
isolates a vertex of $G$. The super connectivity of a graph  have
received the attention of many researchers(see, for example, [1],
[4], [9- 10], [13- 14]).

The lexicographic product of two graphs $G_{1}=(V(G_{1}), E(G_{1}))$
and
 $G_{2}=(V(G_{2}), E(G_{2}))$, which is denoted by  $G_{1}\circ G_2$, is the graph
with vertex set $V(G_{1})\times V(G_{2})$ such that two vertices
$(x_{1}, y_{1})$ and $(x_{2}, y_{2})$ are adjacent if and only if
either $x_{1}x_{2}\in E(G_{1})$ or $x_1=x_2$ and $y_{1}y_{2}\in
E(G_{2})$. According to [6], the lexicographic product of two graphs
first was defined in [5]. Note that in the sense of isomorphism  the
lexicographic product does not satisfies the commutative law. Many
graph theoretical properties  of lexicographic product of graphs
have been studied in the literature (see, for example [7- 8]).

It is easy to check that $G_{1}\circ G_2$ is connected if and only
if $G_1$ is connected. The connectivity of lexicographic product of
two  graphs has been studied in [13].

This paper is devoted to the $k_{1}$-connectivity and super
connectivity of lexicographic product graphs.  In Section 2 we  recall some
results and give some propositions of lexicographic product graphs.
Also, we gain bounds for the $k_{1}$-connectivity of lexicographic
product graphs. In particular, we determine$$k_{1}(G_{1}\circ
G_2)={\rm min}\{ k_{1}(G_1)|V(G_2)|, k(G_1\circ G_2)+|V _{0}(G_1-X)|
|V_{0}(G_2)|\}. $$ where $G_1$ is connected non-complete and  minimum is taken over all
vertex-cut sets  $X$ of $G_1$  such that
 $|V_{0}(G_1-X)|$ has minimum possible cardinality. Moreover,   We investigate the super connectivity
of lexicographic product graphs.

\section{Some properties of lexicographic product }

\begin{thm} ({\rm\cite{YX}})\ \  \label{NC3}
Let $G_1$ be a connected non-complete graph  and $G_2$ be a graph.
Then $k(G_{1}\circ G_2)=k(G_{1})|V(G_{2})|$.

  Moreover, if $G_1=K_n$ then $k(G_{1}\circ G_2)=(n-1)|V(G_{2})|+
k(G_2)$.

\end{thm}
 Let $G_1$ be a connected non-complete graph with $k(G_1)=n$ and a minimum vertex-cut $X=\{x_{1}, \cdot\cdot\cdot,
 x_{n}\}$. By the proof of Theorem 2.1, one conclude that if $V(G_{2})=\{y_{1}, \cdot\cdot\cdot,
 y_{m}\}$, then $$\bar{X}=\{(x_1, y_{1}),\cdot\cdot\cdot, (x_1, y_{m}),
  (x_2, y_{1}),\cdots (x_2, y_m), (x_n, y_{1}), \cdots, (x_n, y_{m}) \}$$ is a minimum vertex-cut in $G_{1}\circ G_{2}$.
  We use these notations for the rest.

 Now, we investigate $k_1$-connectivity of $G_{1}\circ
 G_{2}$.

\begin{thm}
Let $G_1$  be a connected non-complete graph. If $k_{1}(G_1)=k(G_1)$
then $k_{1}(G_{1}\circ G_2)=k(G_{1}\circ G_2) $.
\end{thm}
\begin{proof}
Since $k_{1}(G_1)=k(G_1)$, there exists a minimum vertex-cut of
$G_1$ which is a $k_{1}$-vertex-cut of it. Let $X$ be both minimum
vertex-cut and $k_{1}$-vertex-cut of $G_{1}$. By Theorem 2.1,
 $\bar{X}$ is a minimum vertex-cut of $G_{1}\circ G_2$. Also, each component of $G_{1}-X$ has no isolated vertex. So, each
component of $(G_{1}\circ G_2)-\bar{X}$ has no isolated vertex, that
is, $\bar{X}$ is a $k_{1}$-vertex-cut of $G_{1}\circ G_2$, as needed.

\end{proof}
\begin{thm}
Let $G_1$  be a connected  non-complete graph. If $k(G_1)<
k_{1}(G_1)< \infty$ then
$$k_{1}(G_{1}\circ G_2)={\rm min}\{ k_{1}(G_1)|V(G_2)|, k(G_1\circ
G_2)+|V_{0}(G_1-X)| |V_{0}(G_2)|\} $$ where  minimum is taken over all
vertex-cut sets  $X$ of $G_1$  such that
 $|V_{0}(G_1-X)|$ has minimum possible
cardinality.
\end{thm}
\begin{proof}
Let $V(G_1)=\{x_1, \cdot \cdot \cdot , x_n\}$, $V(G_2)=\{y_1, \cdot
\cdot \cdot, y_m\}$ and $V_0(G_2)=\{\acute{y}_1, \cdot \cdot \cdot,
\acute{y}_t\}$, where $t \geq 0$.
 Let $X=\{\acute{x}_1, \cdot
\cdot \cdot, \acute{x}_k\}$  be a minimum vertex-cut  of $G_1$ that
has the minimum number of neighbors $\{\alpha\}$ with
$N_{G}(\alpha)\subseteq X$. Let $S=\{\grave{x}_1, \cdot \cdot \cdot,
\grave{x}_r\}$ be a minimum $k_1$-vertex-cut of $G_1$. Now if
$\{\breve{x}_1, \cdot \cdot \cdot, \breve{x}_f\}$ be such that
$N_{G_{1}}(\breve{x}_{i})\subseteq X$, then $$\tilde{X}=\{
(\breve{x}_1, \acute{y}_1), \cdot \cdot \cdot, (\breve{x}_1,
\acute{y}_t), \cdot \cdot \cdot, (\breve{x}_f, \acute{y}_1),
\cdot\cdot \cdot, (\breve{x}_f, \acute{y}_t), (\acute{x}_1,
y_1),\cdot\cdot\cdot, (\acute{x}_1, y_m), \cdot\cdot\cdot,
(\acute{x}_k, y_1), \cdot\cdot\cdot, (\acute{x}_k, y_m) \}$$ is an
$k_1$-vertex-cut in $G_1\circ G_2$. Also
$$\bar{S}=\{(\grave{x}_1, y_1), \cdot \cdot \cdot, (\grave{x}_1,
y_m), \cdot \cdot \cdot, (\grave{x}_r, y_1), \cdot \cdot \cdot,
(\grave{x}_r, y_m)\}$$ is an $k_1$-vertex-cut in $G_1\circ G_2$.
Hence, $$k_{1}(G_{1}\circ G_2)\leq {\rm min}\{ k_{1}(G_1)|V(G_2)|,
k(G_1\circ G_2)+|V_{0}(G_1-X)| |V_{0}(G_2)|\}.$$ We show that
$k_{1}(G_{1}\circ G_2)= {\rm min}\{ k_{1}(G_1)|V(G_2)|, k(G_1\circ
G_2)+|V_{0}(G_1-X)| |V_{0}(G_2)|\}.$

Now, let $A$ be a $k_1$-vertex-cut in $G_1\circ G_2$. So $(G_1\circ
G_2)-A$ has at least two components, say,  $C_1$ and $C_2$. Let
$(x_i, y_j)\in C_1$ and $(x_p, y_q)\in C_2$. We have two
cases:\\{\bf Case 1.} Let $x_i\neq x_p$. Clearly, $x_ix_p\not\in
E(G_1)$. Now, by the proof of Theorem 2.1, there are at least
$k(G_1)$ disjoint paths, $P_1, P_2, \cdot\cdot\cdot, P_{k(G_{1})}$
between $x_i$ and $x_p$ in $G_{1}$. So we must
 choose at least one vertex $x_d$ of each path $P_{d}$ and put $\{(x_d, y_j) \mid d\in \mathbb{Z}_{k(G_{1})}, j\in \mathbb{Z}_{m}\}$
 in $A$. Let the number of isolated vertices in $G_1-\{x_1,
\cdot\cdot\cdot, x_{k(G_{1})}\}$ is $l$. note that  $l\geq 0$. Thus
the number of isolated vertices in $(G_1\circ G_2)-\{(x_d, y_j) \mid
d\in \mathbb{Z}_{k(G_{1})}, j\in \mathbb{Z}_{m}\}$ is
$l|V_{0}(G_{2})|$. In the remaining graph we put all isolated
vertices in $A$. Hence $|A|=k(G_1)|V(G_2)|+l|V_{0}(G_{2})|\geq
|\tilde{X}|$ .\\{\bf Case 2.}  Let $x_i= x_p$. So $y_jy_q\not\in
E(G_2)$ and $G_2$ is not totally disconnected. Now $\{(x_u, y_v)
\mid v\in \mathbb{Z}_{m}\}\subset A$ for every $x_u\in
N_{G_{1}}(x_i)$. There exist $y_l, y_{\acute{l}}$ such that $y_jy_l\in
E(G_2)$, $y_qy_{\acute{l}} \in E(G_2)$ and $y_ly_{\acute{l}}\not\in
E(G_2)$. Let there are $t\geq 0$ paths between $y_l$ and
$y_{\acute{l}}$. There is $\acute{t} \leq t$ such that $\{(x_i, y_{z})
\mid z\in \mathbb{Z}_{\acute{t}}\} \subset A$.  Hence in this
case $|A|\geq \delta(G_1)|V(G_2)|+|V_{0}(G_{2})|+\acute{t}\geq |\tilde{X}|$.\\
Therefore $$k_{1}(G_{1}\circ G_2)={\rm min }\{ k_{1}(G_1)|V(G_2)|,
k(G_1\circ G_2)+|V_{0}(G_1-X)| |V_{0}(G_2)|\}.$$
\end{proof}
 By the proof of above theorem we have the following corollary.
\begin{cor}
Let $G_1$  be a connected non-complete graph. If $k_{1}(G_1)=\infty$ then
$$k_{1}(G_{1}\circ G_2)=k(G_{1}\circ G_2)+|V_{0}(G_1-X)|
|V_{0}(G_2)|$$  where  minimum is taken over all
vertex-cut sets  $X$ of $G_1$  such that
 $|V_{0}(G_1-X)|$ has minimum possible cardinality. Also, if  $|V_{0}(G_2)|=0$ then $k_{1}(G_{1}\circ
G_2)=k(G_{1}\circ G_2)$.

\end{cor}

It has been shown that a super connected network is most reliable
and has the smallest vertex failure rate among all the networks with
the same connectivity(see, for example, [11, 12]). In this section
we investigate when $G_{1}\circ G_2$ is super connected.
\begin{thm}
Let $G_1$  be a connected non-complete graph. \\1) If  $G_2$ is
connected then $G_{1}\circ G_2$ is not super connected. \\2) If
$G_2$ is disconnected and $|V_{0}(G_2)|=0$ then $G_{1}\circ G_2$ is
not super connected.\\3) If  $G_2$ is disconnected  with
$|V_{0}(G_2)|\geq 1$ and $G_1$ is super connected, then $G_{1}\circ
G_2$ is super connected.
\end{thm}
\begin{proof}
1) By Corollary 2.4, if $X$ is a minimum vertex-cut of $G_1$ then
$\bar{X}$ is a minimum $k_1$-vertex-cut of $G_1\circ G_{2}$. That
is, $\bar{X}$ can not isolate any vertex of $G_1\circ G_{2}$. Hence
$G_{1}\circ G_2$ is not super connected.\\
 (2) and (3) are hold by the proof of Theorem 2.3.
\end{proof}
The following example shows that in part 3 of Theorem 3.1, if $G_1$
is not super connected then $G_{1}\circ G_2$  is not super connected
as well.
Consider the following graph for $G_1$.

$$
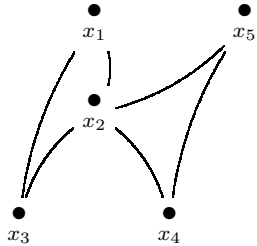
\xygraph{
 !{<0cm, 0cm>;<1cm, 0cm>:<0cm, 1cm>::}
 !{(3, 2.7)}*+{\underset {x_1} \bullet}="x_1"
 !{(2,0)}*+{\underset {x_3} \bullet}="x_3"
 !{(4,0)}*+{\underset {x_4} \bullet}="x_4"
 !{(3,1.5)}*+{\underset {x_2} \bullet}="x_2"
 !{(5,2.7)}*+{\underset {x_5} \bullet}="x_5"
  "x_2" -@/^/@[red]"x_4"
  "x_2" -@/_/@[red]"x_5"
  "x_2"-@/_/@[red]"x_1"
  "x_2"-@/_/@[red]"x_3"
  "x_1"-@/_/@[red]"x_3"
   "x_5"-@/_/@[blue]"x_4"
  }
$$
Clearly $G_1$ is not super connected. Let $V(G_2)=\{y_1, y_2, y_3\}$
where $y_1y_2\in E(G_2)$ and $y_3$ is an isolated vertex. Hence
$\{(x_2, y_1), (x_2, y_2), (x_2, y_3)\}$ is a minimum vertex-cut in
$G_1\circ G_{2}$  and isolates no vertex.




\end{document}